\documentclass[oneside,reqno]{alea2}
\usepackage{mathpazo}
\usepackage[T1]{fontenc}
\usepackage[latin9]{inputenc}
\usepackage{amssymb}
\IfFileExists{url.sty}{\usepackage{url}}
                      {\newcommand{\url}{\texttt}}

\makeatletter
  \theoremstyle{plain}
  \newtheorem{thm}{Theorem}
 \theoremstyle{definition}
 \newtheorem*{defn*}{Definition}
  \theoremstyle{plain}
  \newtheorem*{conjecture*}{Conjecture}

\DeclareMathOperator{\supp}{supp}
\DeclareMathOperator{\Ent}{Ent}
\theoremstyle{plain}
\newtheorem*{question*}{Question}
\def\solra{\mathpalette\solra@}
\def\solra@#1#2{
  \vbox{
    \halign{ 
      ##\crcr 
      \hfil$\scriptscriptstyle\leftrightarrow$\hfil
      \crcr
      \noalign{
        \kern-\ex@\nointerlineskip
        \vskip 0.6mm
      }
      $\m@th\hfil#1#2\hfil$\crcr
    }
  }
}

\makeatother

\begin{document}

\title{On the precision of the spectral profile}

\author{Gady Kozma}

\begin{abstract}
We examine the spectral profile bound of Goel, Montenegro and Tetali
for the $L^{\infty}$ mixing time of continuous-time random walk in
reversible settings. We find that it is precise up to a $\log\log$
factor, and that this $\log\log$ factor cannot be improved.
\end{abstract}


\address{The Weizmann Institute of Science, Rehovot POB 76100, Israel.}

\email{gady.kozma@weizmann.ac.il}


\maketitle

\section{Introduction}

Of all the formulas suggested in the literature as bounds for the
mixing time of a finite graph (see e.g.\ \cite{LK99,MP05,FR07}),
possibly the most promising, from a geometric point of view, is the
spectral profile formula. Introduced by Goel, Montenegro and Tetali
\cite{GMT06}, it brings into the realm of finite graphs the idea
of \emph{Faber-Krahn inequalities}. A Faber-Krahn inequality is an
inequality relating the volume of a set $A$ and the first eigenvalue
of the Laplacian with Dirichlet boundary conditions on $A$ --- we
will give all definitions in the discrete settings below, but for
the history of the topic, mainly in continuous settings, one should
consult \cite{G94}, \cite[\S VIII.6]{C01} or \cite{B01}, which
has a somewhat different take on this topic and an excellent historical
survey. This approach is promising because, as Grigory'an discovered
\cite{G94}, on a general complete manifold it gives sharp estimates
on the decay of the heat kernel, \emph{even in cases where the manifold
does not have polynomial volume growth}. The requirement of polynomial
growth was essential in previous approaches to this problem, using
Sobolev \cite{V85} or Nash \cite{CKS87} inequalities.

Let us describe Faber-Krahn inequalities in the discrete settings.
We will work with weighted, undirected, finite graphs. Let $G$ be
such a graph and $\omega:G\times G\to[0,\infty)$ the weight function.
The \emph{heat kernel} is defined by \[
K(x,y):=\frac{\omega(x,y)}{\omega(x)}\quad\omega(x):=\sum_{z}\omega(x,z).\]
The heat kernel is a stochastic matrix (i.e.\ $\sum_{y}K(x,y)=1$)
and hence describes a Markov chain on $G$. The symmetry $\omega(x,y)=\omega(y,x)$
gives that it is self-adjoint with respect to the stationary measure
$\pi$ defined by \[
\pi(x):=\frac{\omega(x)}{\sum_{y}\omega(y)}\]
and therefore the associated Markov chain is \emph{reversible}. It
is important to note that the results of \cite{GMT06} are not restricted
to the reversible case, and apply to any finite Markov chain, but
in this paper we will restrict our attention to the reversible case.
The Laplacian, which is an operator on $L^{2}(G,\omega)$ is defined
simply as $\Delta:=I-K$ and is self-adjoint and positive.

When $A\subset G$ is some subset, we will introduce the restricted
Laplacian with Dirichlet boundary conditions\begin{equation}
(\Delta_{A}f)(x)=\begin{cases}
\Delta f(x) & v\in A\\
0 & \mbox{otherwise.}\end{cases}\label{eq:Dirichlet}\end{equation}
The smallest eigenvalue for $\Delta_{A}$ will be denoted by $\lambda_{0}(A)$.
It is easy to see that $\lambda_{0}(A)$ may also be defined as \begin{equation}
\lambda_{0}(A)=\inf_{\substack{\supp f\subset A\\
f\not\equiv0}
}\frac{\langle\Delta f,f\rangle}{||f||_{2}^{2}}\label{eq:func}\end{equation}
where $\langle f,g\rangle=\sum_{x}f(x)g(x)\pi(x)$ and $||f||_{p}^{p}=\sum_{x}f(x)^{p}\pi(x)$.
It is somewhat more elegant to describe the results of \cite{GMT06}
with the following quantity instead,\begin{equation}
\lambda(A)=\inf_{\substack{\supp f\subset A\\
f\geq0,f\not\equiv\mbox{const}}
}\frac{\langle\Delta f,f\rangle}{||f||_{2}^{2}-||f||_{1}^{2}}\label{eq:funcA}\end{equation}
and we will adhere to this convention. Note that as long as $\pi(A)\leq1-\epsilon$
the quatities $\lambda_{0}$ and $\lambda$ are comparable \cite[eq.\ (1.4)]{GMT06}.
A Faber-Krahn inequality is an inequality of the form $\lambda(A)\leq\Lambda(\pi(A))$
for some function $\Lambda$, so the minimal function $\Lambda$ satisfying
this is defined by \[
\Lambda(r):=\inf_{0<\pi(A)\leq r}\lambda(A).\]
$\Lambda(r)$ is the \emph{spectral profile}. It is defined for all
$r\geq\pi_{*}:=\min_{\emptyset\neq A\subset G}\pi(A)$.

The main result of \cite{GMT06} is a bound for the $L^{\infty}$
mixing time of the continuous-time random walk in terms of the spectral
profile. Let us give the necessary definitions. The continuous-time
random walk on $G$ is defined using $-\Delta$ as the infinitesimal
generator. Explicitly, we define \[
H_{t}=e^{-t\Delta}=e^{-t}\sum_{n=0}^{\infty}\frac{t^{n}}{n!}K^{n}(x,y)\]
and think about $H_{t}(x,y)$ as the probability that a particle doing
continuous-time random walk on $G$, starting from $x$ will be at
$y$ at time $t$. Hence we define the mixing time using\[
\tau_{\infty}(\epsilon):=\inf\left\{ t>0:\sup_{x,y\in G}\left|\frac{H_{t}(x,y)-\pi(y)}{\pi(y)}\right|\leq\epsilon\right\} .\]
We may now state the main result of \cite{GMT06},\begin{equation}
\tau_{\infty}(\epsilon)\leq\int_{4\pi_{*}}^{4/\epsilon}\frac{2\, dr}{r\Lambda(r)}.\label{eq:goel}\end{equation}
In the rest of the discussion we will fix $\epsilon=\frac{1}{2}$
and denote the left hand side by $\tau_{\infty}$ and the right hand
side by $\rho$.

\subsection{}

The starting point of this short note was the hope that in fact (\ref{eq:goel})
is precise in the sense that $\rho<C\tau_{\infty}$ (%
\footnote{Here and below we use $C$ and $c$ to denote absolute positive constants
that may be different from place to place. $C$ will be used for constants
which are {}``large enough'' and $c$ for constants which are {}``small
enough''. The notation $f\approx g$ will stand for $cf\leq g\leq Cf$. %
}). This was motivated by the fact that Faber-Krahn inequalities give
sharp bounds in many interesting manifolds, and by the fact that (\ref{eq:goel})
is in fact sharp under a certain $\delta$-regularity condition (see
\cite[\S 3]{GMT06}). And in fact, the techniques there give quite
easily (and with no regularity assumption), the following:

\begin{thm}
\label{thm:loglog}For any finite graph $G$, \[
\rho<C\tau_{\infty}\log\log1/\pi_{*}(G).\]

\end{thm}
Unfortunately, it turns out that this cannot be improved. Indeed we
have 

\begin{thm}
\label{thm:sharp}There exist a sequence $n_{k}\to\infty$ and graphs
$G_{k}$ of size $n_{k}$ and $\pi_{*}=\frac{1}{n_{k}}$ such that
\[
\rho\geq c\tau_{\infty}\log\log n_{k}.\]

\end{thm}
The proof of theorem \ref{thm:sharp} is also not difficult --- the
graphs $G_{k}$ will be (details in section \ref{sec:revers}) composed
of $\log\log n_{k}$ pieces $H_{i}$ where each $H_{i}$ corresponds
to a distinct range of $r$-s in the integral defining $\rho$ (\ref{eq:goel}).
On the other hand, a random walker starting at $H_{i}$ will see only
$H_{i}$ --- when it finally leaves $H_{i}$ it will already be too
mixed to notice any effects from the other $H_{j}$-s. Thus, perhaps
the most natural question to ask is 

\begin{question*}Is it possible to have the graphs $G_{k}$ transitive?\end{question*}

A graph $G$ is transitive if for all $x,y\in G$ there exists an
automorphism of the graph taking $x$ to $y$. It would be extremely
exciting if the answer to the question were to be no. A less exciting,
but nonetheless very natural question is as follows. 

\begin{question*}Is it possible to have the graphs $G_{k}$ unweighted
and of uniformly bounded degrees?\end{question*} Here the rationale
for the question is geometric. The analogy between graphs and manifolds
works best for manifolds with bounded geometry and graphs with bounded
degrees. Hence there is a certain discord in the fact that the examples
constructed in theorem \ref{thm:sharp} are weighted. One would be
tempted to solve the question by constructing the graphs (call them
$G_{k}^{\mathrm{simple}}$) randomly, namely put an edge between $x$
and $y$ in $G_{k}^{\mathrm{simple}}$ with probability $\omega(x,y)$
where $\omega$ is the weight function of $G_{k}$. However, more
care is needed --- applying the recipe above naively would immidiately
create logarithmic tails that would dominate the mixing time. 

Finally, we remark that in non-reversible settings existing bounds
are quite weak. For example, it is possible to have $\rho\geq c|G|^{2}$
while $\tau_{\infty}\leq C|G|\log|G|$. A careful discussion of this
phenomenon can be found in \cite[examples 5.3-5.5]{MT06}.

\subsection{}

Another relevant set of problems revolves around the following: is
the mixing time a geometric property? This is particularly interesting
since many results in mixing have been achieved using representation
theory (\cite{BD92} is probably the most famous) or using coupling
(e.g.\ \cite{LRS01}), techniques which are better described as {}``algebraic''
rather than {}``geometric''. To make the question formal let us
define the notion of a rough isometry.

\begin{defn*}
Let $X$ and $Y$ be metric spaces and let $f:X\to Y$ be a function
and let $K\in(0,\infty)$. We say that $f$ is a $K$-rough isometry
if the following two properties hold:
\begin{enumerate}
\item For any $a,b\in X$,\[
\frac{1}{K}d(a,b)-K\leq d(f(a),f(b))\leq Kd(a,b)+K.\]

\item For any $y\in Y$ there exists an $x\in X$ such that\[
d(f(x),y)\leq K.\]

\end{enumerate}
\end{defn*}
To use this for Markov chains we will restrict ourselves to the simplest
settings, that of random walk on a (unweighted) graph with bounded
degree. In this case the graph has a natural metric structure given
by the path metric, i.e.~the distance $d(v,w)$ is defined to be
the length of the shortest path between $v$ and $w$. And we ask:
is the mixing time invariant to rough isometries? Formally:

\begin{conjecture*}
Let $G$, $H$ be two graphs with $\deg G,\deg H\leq d$. Let $f:G\to H$
be a $K$-rough isometry in the path metrics on $G$ and $H$. Then\begin{equation}
\tau(G)\leq C(K,d)\tau(H).\label{eq:tauGtauH}\end{equation}

\end{conjecture*}
Since a rough isometry is reversible, this would in fact imply that
$\tau(G)\approx\tau(H)$.

It is an interesting observation that all approximations for the mixing
time I am aware of are rough isometry invariants. It is easy to see
that isoperimetric inequalities are rough-isometry invariants, and
hence both the Lov\'asz-Kannan integral \cite{LK99} and the Fountoulakis-Reed
integral \cite{FR07} (which bounds the $L^{1}$ mixing time rather
than our $\tau_{\infty}$, but the conjecture is just as relevant
for $\tau_{1}$) are rough-isometry invariants. To see that, for example,
the spectral gap is a rough isometry invariants one has to define
it using functional inequalities i.e.\ (\ref{eq:funcA}) --- note
that the spectral gap is exactly $\lambda(G)$ --- and then it becomes
easy to check that the spectral gap and the spectral profile are both
rough isometry invariants. Thus a \emph{precise} bound of this style
for the spectral gap would probably imply the conjecture.

In particular, combining \cite[theorem 1.1]{GMT06} with theorem \ref{thm:loglog}
gives a weaker form of (\ref{eq:tauGtauH}): \begin{equation}
\tau_{\infty}(G)\leq C(K,d)(\log\log|G|)\tau_{\infty}(H).\label{eq:riloglog}\end{equation}
This result, however, is not new. Indeed, $\tau_{\infty}$ is comparable
to the best constant in the logarithmic Sobolev inequality $\alpha$
defined by \[
\alpha=\inf_{\Ent_{\pi}f^{2}\neq0}\frac{\langle\Delta f,f\rangle}{\Ent_{\pi}f^{2}}\]
in the sense that \[
\frac{c}{\alpha}\leq\tau_{\infty}\leq\frac{C\log\log1/\pi_{*}}{\alpha}.\]
See \cite{MT06} for historical background, the definition of the
entropy $\Ent_{\pi}$ and for the equivalence (theorem 4.13 ibid).
It is easy to see that $\alpha$ is a rough isometry invariance hence
this gives another derivation of (\ref{eq:riloglog}).

We end this discussion with an observation of Itai Benjamini, that
the mixing time \emph{from a given point} is not a rough isometry
invariant. Thus, for example, the mixing time from the root of a binary
tree of height $h$ is $\approx h$. However, the mixing time from
a neighbor of the root is $\approx2^{h}$ (see \cite[chapter 5]{AF}
for both). Since there is a rough isometry of a tree on itself carrying
the root to a neighbor, this demonstrates the claim.

I wish to thank L\'aszl\'o Lov\'asz and Prasad Tetali for many useful
discussions. This material is partially based upon work supported
by the National Science Foundation under agreement DMS-0111298. Any
opinions, findings and conclusions or recommendations expressed in
this material are mine and do not necessarily reflect the views of
the National Science Foundations

\section{\label{sec:revers}Proofs}

\begin{proof}
[Proof of theorem \ref{thm:loglog}]Denote by $A_{k}$ a Rayleigh
set of measure $2^{-k}$ i.e.~$\pi(A_{k})\leq2^{-k}$ and\[
\lambda(A_{k})=\min\{\lambda(S):\pi(S)\leq2^{-k}\}.\]
Where $\lambda$ is from (\ref{eq:funcA}). It is easy to see that
\[
\rho\approx\sum_{k=1}^{\left\lfloor \log_{2}1/\pi_{*}\right\rfloor }\frac{1}{\lambda(A_{k})}.\]
On the other hand, by \cite[lemma 3.1]{GMT06}, for any $k\geq2$\begin{equation}
\tau_{\infty}\geq c\frac{-\log(\pi(A_{k}))}{\lambda(A_{k})}\geq\frac{ck}{\lambda(A_{k})}.\label{eq:tauklam}\end{equation}
As for $k=1$, we have $1/\lambda(A_{1})\leq1/\lambda(G)$ but $\lambda(G)$
is just the spectral gap and hence $1/\lambda(G)\leq C\tau_{\infty}$
\cite[theorem 4.9]{MT06}. Hence (\ref{eq:tauklam}) holds for $k=1$
as well. Therefore\[
\rho\leq\sum_{k=1}^{\left\lfloor \log_{2}1/\pi_{*}\right\rfloor }\frac{C\tau}{k}\approx\tau\log\log1/\pi_{*}.\qedhere\]

\end{proof}

\begin{proof}
[Proof of theorem \ref{thm:sharp}]We may assume w.l.o.g.~that $k$
is sufficiently large. We define simply \[
n_{k}=(k-\left\lceil \log k\right\rceil +1)2^{2^{k}}\]
where $\log$, here and below, is to base $2$. The graph will consist
of $k-\left\lceil \log k\right\rceil +1$ pieces, which we denote
by $H_{\left\lceil \log k\right\rceil },\dotsc,H_{k}$, each with
$2^{2^{k}}$ vertices. We can already define the weight function between
the $H_{l}$s: for every $v_{1},v_{2}$ in different $H_{l}$s we
set \[
\omega(v_{1},v_{2})=\frac{k2^{-k}}{|G|}.\]
Let $A_{l}$ be a set of vertices of size $2^{2^{k}-2^{l}}$ and $B_{l}$
a set of size $2^{2^{l}}$. Setwise we define $H_{l}=A_{l}\times B_{l}$
and then define the weight function $\omega$ as follows:\[
\omega((a_{1},b_{1}),(a_{2},b_{2}))=\begin{cases}
\frac{1}{|H_{l}|}2^{l-k}+\frac{1}{|G|}k2^{-k} & b_{1}\neq b_{2}\\
\frac{1}{|A_{l}|}+\frac{1}{|H_{l}|}2^{l-k}+\frac{1}{|G|}k2^{-k} & b_{1}=b_{2}.\end{cases}\]
With this definition of $\omega$ we would have that $\omega(v)=1+2^{l-k}+k2^{-k}$
for every $v\in H_{l}$. The inhomogeneity of $\omega$ is somewhat
bothersome so we modify $\omega(v,v)$ to fix this, writing\[
\omega(v,v)=(1-2^{l-k})+\frac{1}{|A_{l}|}+\frac{1}{|H_{l}|}2^{l-k}+\frac{1}{|G|}k2^{-k}\quad\forall v\in H_{l}\]
with the result being that $\omega(v)=2+k2^{-k}$ for all $v$. 

With our graph $G$ defined we can start investigating its properties.
We first estimate the spectral profile $\rho$. By the discrete inverse
Cheeger inequality \cite[lemma 2.1]{AM85}, for any set $S$, \[
\lambda(S)\leq C\frac{\pi(\partial S)}{\pi(S)}=C\frac{\omega(\partial S)}{\omega(S)}.\]
where we consider the weight function $\omega$ as a measure which
is a constant multiple of $\pi$. We use it for the set $A_{l}\times\{\mbox{pt}\}$
which we confusingly call $\tilde{A_{l}}$. Now, $\omega(\tilde{A_{l}})\approx|A_{l}|=2^{2^{k}-2^{l}}$.
There are two types of edges coming out of $\tilde{A_{l}}$, edges
to $H_{l}$ and edges to the other $H_{i}$s. The first type has weight
\[
\frac{1}{|H_{l}|}2^{l-k}+\frac{1}{|G|}k2^{-k}=\frac{1}{|H_{l}|}2^{l-k}(1+o(1))\]
where the $o$ notation above and also below means {}``as $k\to\infty$,
uniformly in $l\in[\log k,k]$''. so, after summing over all couples
$(v_{1},v_{2})$, $v_{1}\in\tilde{A_{l}}$ and $v_{2}\in H_{l}\setminus\tilde{A_{l}}$
gives a total contribution $\leq|A_{l}|2^{l-k}(1+o(1))$. The second
type has weight $\frac{1}{|G|}k2^{-k}$ so after summing over all
$(v_{1},v_{2})$, $v_{1}\in\tilde{A_{l}}$ and $v_{2}\in G\setminus H_{l}$
gives a total contribution $\leq|A_{l}|k2^{-k}$. Since $l>\log k$
we get\begin{equation}
\lambda(\tilde{A_{l}})\leq C\frac{\omega(\partial\tilde{A_{l}})}{\omega(\tilde{A_{l}})}\leq C2^{l-k}.\label{eq:lamAl}\end{equation}
Now, $\pi(\tilde{A_{l}})=|A_{l}|/|G|=1/(k-\left\lceil \log k\right\rceil +1)2^{2^{l}}$.
For brevity denote $\epsilon=1/(k-\left\lceil \log k\right\rceil +1)$.
We get that, \[
\int_{\epsilon/2^{2^{l}}}^{\epsilon/2^{2^{l-1}}}\frac{dv}{v\Lambda(v)}\geq\frac{1}{\Lambda(\epsilon/2^{2^{l}})}\int_{\epsilon/2^{2^{l}}}^{\epsilon/2^{2^{l-1}}}\frac{dv}{v}\geq\frac{1}{\lambda(\tilde{A_{l}})}\cdot c2^{l}\stackrel{(\ref{eq:lamAl})}{\geq}c2^{k}.\]
Summing we get\[
\rho=\int_{4\pi_{*}}^{8}\frac{dv}{v\Lambda(v)}\geq c2^{k}(k-\left\lceil \log k\right\rceil +1)\geq ck2^{k}.\]
The proof will be finished once we show that $\tau\leq C2^{k}$.

Let us therefore investigate the random walk on $G$. It will be convinient
to represent it as follows. Assume the walker is at a vertex $v\in H_{l}$.
We first throw a coin which has probability $k/2^{k}\omega(v)$ of
success. Call the event that this throw succeeded $\xi_{1}$ and in
this case choose one of the vertices of $G$ randomly with equal probability
and move there. If $\xi_{1}$ did not occur, throw a second coin which
has probability \[
\frac{2^{l-k}}{\omega(v)-k2^{-k}}=2^{l-k-1}\]
to succeed. Call the event that this throw succeeded $\xi_{2}$ and
in this case choose one of the vertices of $H_{l}$ randomly with
equal probability and move there. Finally, throw a coin with probability
\[
\frac{1}{\omega(v)-k2^{-k}-2^{l-k}}=\frac{1}{2-2^{l-k}}\]
to succeed (if $l=k$ it always does). Call the event that this throw
succeeded $\xi_{3}$ and in this case choose one of the vertices of
the copy of $A_{l}$ containing $v$ randomly with equal probability
and move there. If none of $\xi_{1}$, $\xi_{2}$ and $\xi_{3}$ succeeded,
stay at $v$. It is easy to see that this is equivalent to the walk
on the graph (in fact, we defined the weights on the graph with this
representation in mind).

Examine a random walk of length $2^{k+1}$, starting from some $v\in H_{l}$.
Let $w\in G$ and examine $\mathbb{P}(R(2^{k+1})=w)$ (that starting
point will always be $v$ --- we will not remind this fact in the
notation). We first note that after an event of type $\xi_{1}$ the
walk is completly mixed. Define $\tau_{1}$ to be the first time when
$\xi_{1}$ occurred, which is a stopping time. Using the strong Markov
property we get, for every $t\leq2^{k+1}$, \[
\mathbb{P}(\{\tau_{1}=t\}\cap\{R(2^{k+1})=w\})=\mathbb{P}(\tau_{1}=t)\frac{1}{|G|}\]
and summing over $t$ gives \begin{equation}
\mathbb{P}(\{\tau_{1}\leq2^{k+1}\}\cap\{R(2^{k+1})=w\})=\mathbb{P}(\tau_{1}\leq2^{k+1})\frac{1}{|G|}\leq\frac{1}{|G|}.\label{eq:tau1w}\end{equation}
Now, the event $\tau_{1}>2^{k+1}$ can be estimated simply using\begin{equation}
\mathbb{P}(\tau_{1}>2^{k+1})=\left(1-\frac{k2^{-k}}{2+k2^{-k}}\right)^{2^{k+1}}\leq\left(1-\frac{k2^{-k}}{3}\right)^{2^{k+1}}\leq e^{-2k/3}\label{eq:tau1small}\end{equation}
which immediately gives a lower bound $\mathbb{P}(R(2^{k+1})=w)\geq(1-o(1))/|G|$
valid for all $w$. Further, if $w\not\in H_{l}$ then (\ref{eq:tau1w})
gives an upper bound, since one cannot reach from $v$ to $w$ without
a $\xi_{1}$ event. Hence we will henceforth assume $w\in H_{l}$
and $\tau_{1}>2^{k+1}$. The estimate (\ref{eq:tau1small}) is nice,
but far from our goal of $\frac{3}{2|G|}$.

After an event of type $\xi_{2}$ the walk is totally mixed in $H_{l}$.
Therefore if we define $\tau_{2}$ to be the first time when $\xi_{2}$
occurred then a similar calculation to the above shows that \begin{multline*}
\mathbb{P}(\{\tau_{1}>2^{k+1}\}\cap\{\tau_{2}\leq2^{k+1}\}\cap\{R(2^{k+1})=w\})=\\
\mathbb{P}(\tau_{1}>2^{k+1})\mathbb{P}(\tau_{2}\leq2^{k+1}\,|\,\tau_{1}>2^{k+1})\frac{1}{|H_{l}|}\stackrel{(\ref{eq:tau1small})}{\leq}e^{-2k/3}\frac{1}{|H_{l}|}=o\left(\frac{1}{|G|}\right).\end{multline*}
 With (\ref{eq:tau1w}) we have\begin{equation}
\mathbb{P}(\{\min\{\tau_{1},\tau_{2}\}\leq2^{k+1}\}\cap\{R(2^{k+1})=w\})\leq\frac{1}{|G|}(1+o(1)).\label{eq:tau12w}\end{equation}
Again we note the probability that $\tau_{2}>2^{k+1}$:\begin{equation}
\mathbb{P}(\min\{\tau_{1},\tau_{2}\}>2^{k+1})\leq\left(1-2^{l-k-1}\right)^{2^{k+1}}\leq e^{-2^{l}}\label{eq:tau12}\end{equation}
For the last part we assume $w\in\tilde{A_{l}}$ where here $\tilde{A_{l}}$
is the copy of $A_{l}$ containing $v$. We define $\tau_{3}$ as
the first time $\xi_{3}$ occurred and get \begin{multline}
\mathbb{P}(\{\min\{\tau_{1},\tau_{2}\}>2^{k+1}\}\cap\{\tau_{3}\leq2^{k+1}\}\cap\{R(2^{k+1})=w\})=\\
\mathbb{P}(\{\min\{\tau_{1},\tau_{2}\}>2^{k+1}\})\mathbb{P}(\tau_{3}\leq2^{k+1}\,|\,\min\{\tau_{1},\tau_{2}\}>2^{k+1})\frac{1}{|A_{l}|}\stackrel{(\ref{eq:tau12})}{\leq}\\
\leq e^{-2^{l}}\cdot2^{2^{l}-2^{k}}=o\left(\frac{1}{|G_{k}|}\right).\label{eq:tau123w}\end{multline}
Finally, in the case that $\tau_{1},\tau_{2},\tau_{3}>2^{k+1}$ (so
$w$ must be $v$) we definitely have \begin{equation}
\mathbb{P}(\min\{\tau_{1},\tau_{2},\tau_{3}\}>2^{k+1})\leq\left(1-\frac{1}{2-2^{l-k}}\right)^{2^{k+1}}\leq2^{-2^{k+1}}=o\left(\frac{1}{|G_{k}|}\right).\label{eq:tau4w}\end{equation}
Summing up (\ref{eq:tau1w}), (\ref{eq:tau12w}), (\ref{eq:tau123w})
and (\ref{eq:tau4w}) we finally get \[
\mathbb{P}^{v}(R(2^{k+1})=w)\leq\frac{1}{|G|}(1+o(1))\]
 and hence for $k$ sufficiently large, $\tau\leq2^{k+1}$. This ends
the theorem.
\end{proof}

\end{document}